\documentclass[leqno,12pt]{amsart}
\usepackage{jdr-style,jdr-macros}
\usepackage[all,tips]{xy}
\usepackage{longtable}
\usepackage{url}
\usepackage{verbatim}
\usepackage{tikz}
\usepackage{mathtools}

\CompileMatrices

\newcommand{\parref}[1]{\textup{(\ref{#1})}}

\title{Non-Archimedean and tropical theta functions}

\author{Tyler Foster}
\email{foster@ihes.fr}
\address{L'Institut des Hautes \'Etudes Scientifiques, Le Bois-Marie, 35 route de Chartres, 91440 Bures-sur-Yvette, France}

\author{Joseph Rabinoff} 
\email{rabinoff@math.gatech.edu}
\address{School of Mathematics, Georgia Institute of Technology, Atlanta GA 30332-0160, USA}

\author{Farbod Shokrieh}
\email{farbod@math.cornell.edu}
\address{Department of Mathematics, Cornell University, Ithaca, New York 14853-4201, USA}

\author{Alejandro Soto}
\email{alejandro.soto@uni-tuebingen.de}
\address{Eberhard Karls Universit\" at Tübingen, Fachbereich Mathematik, Auf der Morgenstelle 10, D-72076 T\" ubingen, Germany. }

\begin{document}

\maketitle

\begin{abstract}
  We define a tropicalization procedure for theta functions on abelian varieties over a non-Archimedean field.  We show that the tropicalization of a non-Archimedean theta function is a tropical theta function, and that the tropicalization of a non-Archime\-dean Riemann theta function is a tropical Riemann theta function, up to scaling and an additive constant.  We apply these results to the construction of rational functions with prescribed behavior on the skeleton of a principally polarized abelian variety.  We work with the Raynaud--Bosch--L\"utkebohmert theory of non-Archimedean theta functions for abelian varieties with semi-abelian reduction.
\end{abstract}

\section{Introduction}
\label{sec:intro}
Let $K$ be a nontrivially valued, complete, algebraically closed, non-Archimedean field.  Let $A$ be an abelian variety over $K$.
Suppose for the moment that $A$ is principally polarized with totally degenerate
reduction.  In this case, the non-Archimedean uniformization theory of $A$
essentially amounts to a finitely generated, free abelian group $M$, and a
symmetric pairing $t(\scdot,\scdot)\colon M\times M\to K^\times$, such
that $-\log|t|$ is positive-definite.  This pairing defines
a homomorphism from $M$ to the torus $\bT = \Spec K[M]$ by the rule
$\iota\colon u\mapsto x_u$, where $x_u(\chi^{u'}) = t(u,u')$.  (Here
$\chi^{u'}\colon\bT\to\bG_m$ is the character corresponding to $u'\in M$).  Then
the Berkovich analytification $A^\an$ is canonically isomorphic to the quotient $\bT^\an/\iota(M)$.
In this setting, a (non-Archimedean) Riemann theta function for $A$ is the analytic function $f\colon\bT^\an\to\bA^{1,\an}$ given by the convergent power series
\begin{equation*}\tag{*}
  f = \sum_{u\in M} \sqrt{t(u,u)}\,\chi^u
\end{equation*}
with respect to some choice of $\sqrt t$.

Berkovich~\cite{berkovic90:analytic_geometry} showed that the analytification
$A^\an$ contains a canonical subset $\Sigma$, called a \emph{skeleton}, onto
which it deformation retracts.  This skeleton is a
\emph{tropical abelian variety} (Definition~\ref{def:trop.av},
Proposition~\ref{prop:skel.polarization}): it is a real torus with an integral
structure that admits a ``polarization''.  A principal polarization of $A$
induces a principal polarization of $\Sigma$, in which case $\Sigma$ amounts to
the data of a finitely generated, free abelian group $M$ equipped with a
positive-definite, symmetric bilinear pairing
$[\scdot,\scdot]\colon M\times M\to\R$.  This pairing defines an injective
homomorphism $\iota\colon M\to N_\R = \Hom(M,\R)$ by
$\iota(u) = (u'\mapsto [u,u'])$; the real torus in question is
$\Sigma = N_\R/\iota(M)$.  The tropical Riemann theta function for $\Sigma$ is the piecewise linear function $\phi\colon N_\R\to\R$ defined by the formula
\begin{equation*}\tag{**}
  \phi(v) = \min_{u\in M}\big\{ \tfrac 12[u,u] + \angles{u,v} \big\},
\end{equation*}
where $\angles{\scdot,\scdot}\colon M\times N_\R\to\R$ is the evaluation pairing.

In this paper, we make the obvious analogy between~(*) and~(**) into a precise relationship.  We do not assume $A$ is principally polarized, and more importantly, we allow $A$ to have arbitrary mixed reduction type.  It is still true that $A^\an$ admits a canonical skeleton $\Sigma = N_\R/M'$, which is a tropical abelian variety, but the theory of non-Archimedean theta functions in this case is very much more technical.
We define a tropicalization procedure $f\mapsto f_{\trop}$, which takes a theta function $f$ on $A^\an$ and produces a piecewise-linear function $f_{\trop}\colon N_\R\to\R$.  We prove the following results:

\begin{oneoffthm*}{Theorem~A}
  The tropicalization of a non-Archimedean theta function is a tropical theta function.
\end{oneoffthm*}

See Definition~\ref{def:trop.theta} for the definition of a general tropical theta function, and see Theorem~\ref{thm:trop.theta} for a precise statement of Theorem~A.

\begin{oneoffthm*}{Theorem~B}
  If $f$ is the Riemann theta function associated to a principal polarization of $A$, then $f_{\trop}$ is the tropical Riemann theta function associated to the induced principal polarization of $\Sigma$, up to translation and an additive constant.
\end{oneoffthm*}

See Theorem~\ref{thm:trop.riemann.theta} for a precise statement.

The main interest in Theorems~A and~B is to construct rational functions on $A$ with a prescribed behavior on $\Sigma$.  Given a nonzero rational function $f$ on $A$, one can simply restrict $-\log|f|$ to $\Sigma$, to obtain a piecewise linear function $f_{\trop}\colon\Sigma\to\R$.  Let $(f_1,\ldots,f_n)$ be a tuple of rational functions, and let $f\colon A\dashrightarrow\bG_m^n$ be the rational map $(f_1,\ldots,f_n)$.  Composing with the tropicalization map $\trop\colon\bG_m^{n,\an}\to\R^n$, defined on points by $(x_1,\ldots,x_n)\mapsto-(\log|x_1|,\ldots,\log|x_n|)$, gives a partially-defined function $A^\an\to\R^n$.  The restriction of this function to $\Sigma$ is the function $f_{\trop}\colon\Sigma\to\R^n$ given by $f_{\trop}(x) = (f_{1,\trop}(x),\ldots,f_{n,\trop}(x))$.  Careful construction of $(f_1,\ldots,f_n)$ yields a map $f_{\trop}$ with nice properties (for example, unimodularity); this is the subject of future work of the authors.

In order to allow such constructions, we prove the following result.

\begin{oneoffthm*}{Theorem~C}
  Let $\phi_1,\phi_2$ be tropical theta functions for $\Sigma$ with the same automorphy factor.  Suppose that $\phi_1$ and $\phi_2$ are tropicalizations of non-Archimedean theta functions.  Then $\phi_1-\phi_2\colon\Sigma\to\R$ is the tropicalization of a nonzero rational function on $A$.
\end{oneoffthm*}

See Theorem~\ref{thm:same.automorphy} for a precise statement.  In~\S\ref{sec:constr-ratl-func} we give concrete examples of constructions of rational functions on $A$ using Theorem~C.

\subsection{Notation}\label{par:general.notation}
The following notations are used throughout the paper.  We let $K$ denote an
algebraically closed field which is complete with respect to a nontrivial,
non-Archimedean absolute value $|\scdot|$.  Let $\val = -\log |\scdot|$ denote a
corresponding valuation, $R$ the ring of integers of $K$, and $k$ its
(algebraically closed) residue field.

We will generally use $M$ for a finitely-generated, free abelian group; we
denote its dual by $N = \Hom(M,\Z)$, and we set
$N_\R = \Hom(M,\R) = N\tensor_\Z\R$.  The evaluation pairing is denoted
$\angles{\scdot,\scdot}\colon M\times N_\R\to\R$.  The monoid ring on
$M\cong\Z^n$ is written $K[M]\cong K[T_1^{\pm 1},\ldots,T_n^{\pm 1}]$, and the
character on $\Spec(K[M])$ corresponding to $u\in M$ is written
$\chi^u\in K[M]$.

If $X$ is a finitely-type $K$-scheme, we denote by $X^\an$ its analytification in the sense of Berkovich~\cite{berkovic90:analytic_geometry}.  For
$x\in X^\an$ we let $\sH(x)$ denote the completed residue field at $x$.  This is
a complete valued field extension of $K$ with valuation ring $\sH(x)^\circ$.

Let $L$ be a line bundle on a $K$-scheme $X$ and let $x\in X(K)$ be a $K$-point.
The $x$-fiber of $L$ is denoted $L_x$, and the constant line bundle on $L_x$ is
$L(x)\coloneq L_x\times X$.

The special and generic fibers of an $R$-scheme $\fX$ are denoted $\fX_s$ and
$\fX_\eta$, respectively.

Let $A$ be an abelian variety.  The dual of $A$ is denoted $A'$, and if
$\phi\colon A\to B$ is a homomorphism of abelian varieties then its dual is
denoted $\phi'\colon B'\to A'$.  For $x\in A(K)$ we let $T_x\colon A\to A$ be
translation by $x$.
A line bundle $L$ on $A$ determines a symmetric
homomorphism $\phi_L\colon A\to A'$ defined on points by
$x\mapsto T_x^*L\tensor L\inv$.

\subsection*{Acknowledgments}
This project was completed as a part of the Mathematics Research Communities
workshop held by the American Mathematical Society in the summer of 2013.  The
authors would like to thank the AMS for their support, and Matthew Baker and Sam
Payne for organizing the workshop and for useful discussions.  The third author would like to thank Alberto Bellardini for many helpful conversations.
The authors gratefully thank the anonymous referee for carefully reading the paper and for providing some very helpful comments.

Foster was supported by NSF RTG grant DMS-0943832 and by Le Laboratoire d'Excel\-lence CARMIN. Rabinoff was supported by NSF DMS-1601842. Soto would like to thank to the KU Leuven  and to the Goethe Universit\"at Frankfurt am Main for the great working conditions.

%%%%%%%%%%%%%%%%%%%%%%%%%%%%%%%%%%%

\vskip 1cm

\section{Tropical abelian varieties and tropical theta functions}
\label{sec:tropical.avs}

In this section we fix our notions regarding real tori with integral structure,
tropical abelian varieties, and tropical theta functions.  We define the
tropical Riemann theta function associated to a principally polarized tropical abelian variety.

\subsection{Real tori with integral structure}
We begin by defining real tori with integral structure, a weaker notion than
a tropical abelian variety which is sufficient to define piecewise linear functions.

\begin{defn}\label{def:integral.torus}
  Let $N$ be a finitely generated, free abelian group, let
  $N_\R = N\tensor_\Z\R$, and let $\Lambda\subset N_\R$ be a (full rank) lattice.
  The quotient $\Sigma = N_\R/\Lambda$ is called a
  \emph{real torus with integral structure}.
\end{defn}

The ``integral structure'' in $\Sigma$ is the choice of lattice $N\subset N_\R$,
which need not coincide with the quotient lattice $\Lambda$.

\begin{defn}\label{defn:unimodularity}
  Let $N,N'$ be finitely generated, free abelian groups and let
  $N_\R = N\tensor_\Z\R$ and $N'_\R = N'\tensor_\Z\R$.  Let
  $\Lambda\subset N_\R$ be a lattice and let $\Sigma = N_\R/\Lambda$, with
  quotient homomorphism $\pi\colon N_\R\to\Sigma$.
  \begin{enumerate}
  \item A homomorphism $\phi\colon N_\R\to N'_\R$ is \emph{integral affine} provided
    that $\phi = \psi_\R + v$, where $\psi\colon N\to N'$ is a homomorphism,
    $\psi_\R$ is the extension of scalars of $\psi$, and $v\in N_\R'$.
  \item A continuous function $\phi\colon N_\R\to N_\R'$ is
    \emph{piecewise integral affine} provided that $N_\R$ can be covered by
    full-dimensional polyhedra $\Delta$ such that $\phi|_\Delta$ is integral
    affine.
  \item A continuous function $\bar\phi\colon\Sigma\to N_\R'$ is
    \emph{piecewise integral affine} provided that the composition
    $\phi = \bar\phi\circ\pi$ is piecewise integral affine.
  \end{enumerate}
  We say that a function $\phi$ from $N_\R$ or $\Sigma$ to $\R^r$ has one of the
  above properties if it has that property with respect to the integral structure
  $N' = \Z^r \subset \R^r$.
\end{defn}

Some of the functions $\phi\colon\Sigma\to N_\R'$ arising in the sequel will
be symmetric with respect to negation, so we make the following definitions.

\begin{defn}\label{def:tropically.kummer}
  With the notation in Definition~\ref{defn:unimodularity}, let
  $\phi\colon\Sigma\to N_\R'$ be a piecewise integral affine function.
  We say that $\phi$ is \emph{Kummer} provided that $\phi(x) = \phi(-x)$
  for all $x\in\Sigma$.
\end{defn}

In other words, a Kummer map $\phi\colon \Sigma\to N_\R'$ factors through the
quotient $\Sigma\surject\Sigma/(-1)$ by the negation action.

\subsection{Tropical abelian varieties}
Now we discuss tropical abelian varieties.  These are real tori with integral
structure which admit a polarization, in the sense defined below.  The skeleton of an
abelian variety is naturally a tropical abelian variety: see
Proposition~\ref{prop:skel.polarization} below.  Tropical abelian varieties and
tropical theta functions have been implicit in the literature for years; see for
instance Mumford~\cite[Proposition~6.7]{mumford72:analytic_degen_avs} and
Faltings--Chai~\cite[Definition~1.5,
p.197]{faltings_chai90:degenerat_abelian_varieties}.  The name ``tropical
abelian variety'' seems to have been used first
in~\cite{mikhalkin_zharkov08:tropical_curves}.

\begin{defn}\label{def:trop.av}
  We fix the following data:
  \begin{enumerate}
  \item[(a)] Finitely generated free abelian groups $M,M'$ of the same rank.
  \item[(b)] A nondegenerate pairing $[\scdot,\scdot]\colon M'\times M\to\R$.
  \end{enumerate}
  The pairing defines $M'$ (resp.\ $M$) as a lattice in $N_\R = \Hom(M,\R)$
  (resp.\ $N_\R' = \Hom(M',\R)$).  Suppose that there exists
  $\lambda\colon M'\to M$ such that
  $[\scdot,\lambda(\scdot)]\colon M'\times M'\to\R$ is symmetric and
  positive-definite.  Then the map $\phi\colon N_\R\to N_\R'$ dual to
  $\lambda$ takes $M'$ into $M$ via $\lambda$.  Under these conditions:
  \begin{enumerate}
  \item We say that $\Sigma = N_\R/M'$ is a \emph{tropical abelian variety}.
  \item The \emph{dual} of $\Sigma$ is the tropical abelian variety
    $\Sigma' = N_\R'/M$.
  \item The induced homomorphism $\bar\phi\colon\Sigma\to\Sigma'$ is called a
    \emph{polarization}.
  \item We say that $\phi$ is a \emph{principal polarization}
  if $\lambda$ is an isomorphism.
  \end{enumerate}
\end{defn}

The existence of $\lambda$ here plays the role of Riemann's period relations,
which guarantee that a complex torus is in fact an algebraic variety.  We leave
it as an exercise to show that $\Sigma'$ is in fact a tropical abelian variety.
Since $N = \Hom(M,\Z)$ is a lattice in $N_\R$, a tropical abelian variety is a
real torus with integral structure.

\subsection{Tropical theta functions}
Finally we are able to define tropical theta functions. Recall that classically, a theta function is a holomorphic function on a complex vector space $V$ which is quasi-periodic with respect to a full dimensional lattice in $V$. The same is true in the non-archimedean setting. In the tropical side, a tropical theta function will be defined following the classical counterparts, i.e. as a piecewise linear function on a real vector space which is quasi-periodic with respect to a lattice. This quasi-periodicity condition should be thought as the tropical side of the non-archimedean one. This will made precise in \S 4.

\begin{defn}\label{def:trop.theta}
  Let $\Sigma = N_\R/M'$ be a tropical abelian variety, as in
  Definition~\ref{def:trop.av}, and let $\lambda\colon M'\to M$ be a
  homomorphism definining a polarization.  Let $c\colon M'\to\R$ be a function
  satisfying $c(u_1'+u_2')-c(u_1')-c(u_2') = [u_1',\lambda(u_2')]$ for all
  $u_1',u_2'\in M'$.  A \emph{tropical theta function} with respect to
  $(\lambda,c)$ is a piecewise integral affine function $\phi\colon N_\R\to\R$
  satisfying the transformation law
  \[ \phi(v) = \phi(v + u') + c(u') + \angles{\lambda(u'),v} \]
  for all $u'\in M'$ and $v\in N_\R$.
\end{defn}

Let $c$ be a function as in Definition~\ref{def:trop.theta}.  Then
$u'\mapsto c(u')-\frac 12[u',\lambda(u')]$ is a homomorphism $M'\to\R$.  In other words,
$c(u')$ differs from $\frac 12[u',\lambda(u')]$ by a homomorphism.

\begin{defn}\label{def:trop.riemann.theta}
  With the notation in Definition~\ref{def:trop.theta}, suppose now that
  $\lambda$ is an isomorphism.  The \emph{tropical Riemann theta function}
  associated to $\lambda$ is
  \[ \phi(v) \coloneq \min_{u'\in M'}\big\{ \tfrac 12[u',\lambda(u')] +
 \angles{\lambda(u'),v} \big\}. \]
\end{defn}

The tropical Riemann theta function $\phi$ is in fact a tropical theta function
with respect to $(\lambda,c)$ for $c(u') = \frac 12[u',\lambda(u')]$: however,
it is not obvious that $\phi$ is piecewise integral affine or that it satisfies
the stipulated transformation law.  This will  follow from Theorem~\ref{thm:trop.theta}, but this is not a reasonable proof.
See~\cite{mikhalkin_zharkov08:tropical_curves} for a full discussion of the
tropical Riemann theta function.

\section{Non-Archimedean uniformization and theta functions}
\label{sec:uniformization}

Recall that $K$ is a complete and algebraically closed non-Archimedean valued
field with valuation ring $R$.
Let $A$ be an abelian variety over $K$.  A theta function is almost the same as
a global section of a line bundle $L$ on $A$.  The reason for the name is that
theta functions are constructed analytically on the universal cover of
$A^\an$, as in the classical situation over $\C$, where such functions are generally called $\theta$.  This theory of
non-Archimedean uniformizations and theta functions was worked out by Bosch and
L\"utkebohmert.  It is a beautiful theory, but when $A$ has semi-abelian
reduction it is quite technical.  In this section we recall the results
in~\cite{bosch_lutkeboh91:degenerat_abelian_varieties} that we will use, translated from
Bosch--L\"utkebohmert's language of rigid and formal geometry into the language
of Berkovich analytic spaces.  We also prove a fact about Fourier coefficients
of non-Archimedean theta functions in Proposition~\ref{prop:val.riemann.theta}, which will be important in~\S\ref{sec:tropicalization.theta}.

We discuss the totally degenerate case, which is much simpler, as a running
example.  The reader may want to understand this case first.  See also
Faltings--Chai~\cite{faltings_chai90:degenerat_abelian_varieties}, as well as
Fresnel--van der Put~\cite{fresnel_vanderput04:rigid_analytic_geometry} for a
detailed discussion in the totally degenerate case.

If $L\to A$ is a line bundle then $L^\an\to A^\an$ is also a line
bundle, and it follows from analytic GAGA that $H^0(A,L)=H^0(A^\an,L^\an)$.
Moreover, any analytic line bundle on a proper variety is automatically
algebraic.  Hence we will sometimes neglect to distinguish between algebraic and
analytic line bundles on abelian varieties.

\subsection{Metrized line bundles}\label{par:model.metrics}
Let $\pi\colon L\to X$ be a line bundle on an analytic space.  A \emph{metric}
on $L$ is a function $\|\scdot\|\colon L\to\R_{\geq 0}$ which is compatible with
the $\bG_m^\an$-action and is non-constant on fibers: that is, given $x\in X$
and an isomorphism $\pi\inv(x)\cong\bA^{1,\an}_{\sH(x)} = \Spec(\sH(x)[T])^\an$,
for $y\in\pi\inv(x)$ we have $\|y\| = c|T(y)|$ for some $c\in|\sH(x)^\times|$.
We require the metric to be continuous in the sense that for
$U\subset X$ open and a section $s\colon U\to L|_U$, the function
$x\mapsto\|s(x)\|\colon U\to\R_{\geq 0}$ is continuous.

Given a metrized line bundle $(L,\|\scdot\|)$ on $X$ and a morphism
$\phi\colon Y\to X$, there is an obvious metric on $\phi^*L$.  Given two
metrized line bundles $(L_1,\|\scdot\|_1)$ and $(L_2,\|\scdot\|_2)$, there is a
metric $\|\scdot\|$ on $L\tensor L'$ characterized by
$\|s_1\tensor s_2(x)\| = \|s_1(x)\|_1\|s_2(x)\|_2$.

An integral model induces a metric in the following way.
Let $\fL$ be a line bundle on a flat, proper $R$-scheme $\fX$.  Let $L$ be the
generic fiber of $\fL$, i.e., the restriction of $\fL$  to
$X\coloneq\fX \times_R \Spec(K)$.  In this situation we say that $\fL$ is an
\emph{integral model} of $L$ on $\fX$.  The \emph{model metric}
$\|\scdot\|_\fL\colon L^\an\to\R_{\geq0}$ on $L^\an$ associated to the model
$\fL$ is defined as follows.  Let $y\in L^\an$ with image $x\in X^\an$.  By the valuative criterion of
properness, $x$ extends to a unique $\sH(x)^\circ$-point
$x\colon\Spec\sH(x)^\circ\to\fX$.  Choosing any trivialization of $\fL$ in a
neighborhood of the reduction of $x$ gives an isomorphism
$T\colon x^*\fL\isom\bA^1_{\sH(x)^\circ}$.  We set $\|y\|_\fL = |T(y)|$.  This is
well-defined because any other trivialization $T'$ of $x^*\fL$ differs by a unit in
$\sH(y)^\circ$.  See also~\cite[\S3]{gubler10:canonical_measures}.

Suppose now that $B$ is an abelian variety over $K$ with good reduction.  Up to
isomorphism, there is a unique abelian $R$-scheme $\fB$ whose generic fiber is
identified with $B$.  Let $L$ be a rigidified line bundle on $B$ (a line bundle
with a trivialization of its identity fiber).  It follows from~\cite[Lemma~6.1]{bosch_lutkeboh91:degenerat_abelian_varieties} and from the formal GAGA principle~\cite[{\S}I.10]{fujiwara_kato14:fgrI}  that $L$ has a \emph{canonical}
integral model $\fL$ on $\fB$, so there is a canonical model metric on $L^\an$,
which we denote by $\|\scdot\|_L$.

\subsection{Raynaud--Bosch--L\"utkebohmert uniformization}
\label{bosch_luetkebohmert_uniformization}
Let $A$ be an abelian variety over $K$.  Let $p\colon E^\an\to A^\an$ be the
universal cover (in the sense of topology), and choose a base point
$0\in E^\an$ over the identity.  Then $E^\an$ has the unique structure of an
analytic group with identity $0$, and $p$ is a homomorphism.  In fact $E^\an$ is
the analytification of a group scheme $E$, although the homomorphism $p$ is not
algebraic.  The uniformization theory of Raynaud and Bosch--L\"utkebohmert says
that there are two exact sequences
\begin{align}
  \label{raynaud1}
  &0\To \bT\To E \overset{q}{\To} B\To 0 \\
  \label{raynaud2}
  &0\To M' \To E^\an \overset{p}{\To} A^\an \To 0,
\end{align}
where $\bT = \Spec(K[M])$ is a split $K$-torus, $B$ is an abelian variety with
good reduction, and $M'$ is a lattice in $E^\an$ (see below).  Moreover, all of
these data are uniquely determined by $A$, up to isomorphism.  The
sequence~\eqref{raynaud1} is called a \emph{Raynaud extension}.  Both sequences
are often written in a so-called \emph{Raynaud cross}
\[\xymatrix @=.2in{
  &{M'} \ar[d] & \\
  {\bT^\an} \ar[r] & {E^\an}\ar[d]^(.4)p \ar[r]^(.45)q & {B^\an} \\
  & {A^\an} &
}\]

Given $u\in M$, we construct the pushout diagram
\begin{equation}\label{pushout}
\xymatrix @R=.2in{
  0\ar[r] & \bT \ar[r] \ar[d]^{\chi^u} & E \ar[d]^{e_u} \ar[r]^q &B \ar@{=}[d] \ar[r] &0 \\
  0\ar[r] & \bG_m\ar[r]& E_{u} \ar[r]&B \ar[r]&0
}
\end{equation}
A rigidified translation-invariant line bundle on an abelian variety $B$ amounts
to an extension of $B$ by $\bG_m$: given such an extension $L'$, the
$\bG_m$-action makes $L'$ into a $\bG_m$-torsor over $B$, which is then
identified with the complement of the zero section of a line bundle $L$.  This
line bundle is translation-invariant: given $x\in B$, choose a lift $y\in L'$;
then multiplication by $y$ in $L'$ yields an isomorphism $L\isom x^*L$.
Thus $E_u$ is the complement of the zero locus of such a rigidified
translation-invariant line bundle on $B$; we also denote this line bundle (with
zero section included) by $E_u$ by abuse of notation.  Note that we may regard
$e_u$ as a trivialization of the pullback $q^*E_u$.

Now we explain the lattice condition.  Define the \textit{tropicalization map}
$\trop\colon \bT^\an\to N_\R$ by the rule
\[ \angles{u,\,\trop(x)} \coloneq -\log |\chi^u(x)|, \quad u\in M. \]
This extends in a canonical way to a homomorphism $\trop\colon E^\an \To N_\R$
by setting
\begin{equation}\label{eq:trop.E}
  \angles{u,\trop(x)} = -\log\|e_u(x)\|_{E_u},
\end{equation}
where $\|\scdot\|_{E_u}$ is the model metric defined
in~\parref{par:model.metrics} using the fact that $B$ has good reduction.
To say that $M'$ is a \emph{lattice} means that $\trop$ maps $M'$
isomorphically onto an ordinary lattice (of full rank) in the Euclidean space
$N_\R$.

\begin{rem}\label{rem:formal.raynaud}
  The identity fiber of $\trop\colon E^\an\to N_\R$ is a subgroup $A_0$ of
  $E^\an$, and we have a short exact sequence
  \begin{equation}\label{eq:raynaud.formal}
    0 \To \bT_0 \To A_0 \overset{q_0}\To B^\an \To 0
  \end{equation}
  with $\bT_0 = \trop\inv(0)$, the affinoid torus.  We call the
  sequence~\eqref{eq:raynaud.formal} a \emph{formal Raynaud extension} because
  it arises as the generic fiber of a short exact sequence of formal group
  schemes over $R$.  This sequence splits locally in the formal analytic
  topology on $B^\an$: that is, there exists a cover of $B^\an$ by open sets $V$
  which are the generic fiber of a formal affine, such that
  $q_0\inv(V)\cong\bT_0\times V$.  Extending the splitting to
  $q\inv(V)\cong\bT\times V$, the map $\trop\colon q\inv(V)\to N_\R$ is the
  composition of the first projection $\bT\times V\to\bT$ with
  $\trop\colon\bT\to N_\R$.  See~\cite[\S1]{bosch_lutkeboh91:degenerat_abelian_varieties}
  and~\cite[(4.2)]{gubler10:canonical_measures} for details.
\end{rem}

The two extremal cases for the uniformization theory of $A$ are when $A$ has
good reduction, in which case $\bT$ and $M'$ are trivial and $A=B$, and
when $A$ has totally degenerate (i.e.,\ toric) reduction, in which case $B = 0$
and $E = \bT$. It is precisely this later situation which is analogous to the
uniformization of abelian varieties defined over the complex numbers, as we have
$A^\an\cong\bT^\an/M'$.

\subsection{Duality theory}\label{par:unif.dual}
Let $A$ be an abelian variety with dual $A'$.  Much
of~\cite{bosch_lutkeboh91:degenerat_abelian_varieties} is concerned with relating the
uniformizations of $A$ and $A'$.  The end result is that we have the following
Raynaud crosses:
\[\xymatrix @=.2in{
  & {M'} \ar[d] \ar[dr]^(.4){\Phi} & & &
  & M \ar[d] \ar[dr]^(.4){\Phi'} &\\
  {\bT^\an} \ar[r] & {E^\an} \ar[d]_(.4){p} \ar[r]_(.4){q} & {B^\an} & &
  {\bT'^\an} \ar[r] & {E'^\an} \ar[d]_(.4){p'} \ar[r]_(.4){q'} & {B'^\an} \\
  &{A^\an} & & &
  &{A'^\an} & 
}\]
Here $\bT = \Spec(K[M])$ and $\bT' = \Spec(K[M'])$, and $B$ is the dual of $B'$.
The map $\Phi'$ has the property that $\Phi'(u) = E_u$ (regarded as a $K$-point
of $B'$), and likewise for $\Phi$.  Hence for $(u',u)\in M'\times M$ we have
$e_u(u'),e_{u'}(u)\in P_{\Phi(u'),\Phi'(u)}$, the fiber over
$(\Phi(u'),\Phi'(u))$ of the Poincar\'e bundle $P\to B\times B'$.  It turns out
that $e_u(u') = e_{u'}(u)$, and the map
\begin{equation}\label{eq:trivialization.biext}
  (u',u)\mapsto t(u',u)\coloneq e_u(u') = e_{u'}(u)
  \qquad t\colon M'\times M\To P
\end{equation}
is a trivialization of $(\Phi\times\Phi')^*P$, regarded as a biextension of
$M'\times M$.  The valuation
\begin{equation}\label{eq:pairing.1}
  [\scdot, \scdot] \coloneq -\log\|t(\scdot,\scdot)\|_P\colon M'\times M\To \R
\end{equation}
is bilinear and nondegenerate
by~\cite[Proposition~3.4]{bosch_lutkeboh91:degenerat_abelian_varieties}.

We use the notation introduced above for the rest of this section.

\begin{rem}
  In~\cite{bosch_lutkeboh91:degenerat_abelian_varieties} the roles of $M$ and $M'$ are reversed.
  We chose to follow the opposite convention since in their situation, the
  tropicalization map on $E^\an$ would take values in $N_\R'$, which is
  nonstandard from the tropical point of view.
\end{rem}

\begin{eg}\label{eg:toric.1}
  Suppose that $A$ has toric reduction, so $B = 0$ and $E = \bT$.  The
  trivialization $t$ amounts to a bilinear pairing $M'\times M\to K^\times$ such
  that $-\log|t|$ is nondegenerate.  This defines an embedding
  $M'\inject \bT'=\Spec(K[M'])$ by the rule $\chi^{u'}(u) = t(u',u)$, and the
  quotient $\bT'^\an/M'$ is canonically isomorphic to $A'^\an$.  
  See~\cite[Chapter~6]{fresnel_vanderput04:rigid_analytic_geometry}.
\end{eg}

\subsection{Line bundles}
By descent theory, one can study line bundles on $A$ (equivalently, on $A^\an$)
by studying line bundles on $E^\an$ with an $M'$-linearization.  It turns
out that it is enough to consider only line bundles on $E^\an$ which are pulled
back from $B$.

\begin{eg}\label{eg:linearize.Eu}
  Recall from~\eqref{pushout} that we can regard $E_u$ as an extension of $B$ by
  $\bG_m$.  The composite homomorphism 
  \begin{equation}\label{eq:epsilon_u}
    \epsilon_u\colon~ M'\To E(K) \overset{e_u}\To E_u(K)
  \end{equation}
  defines an $M'$-linearization of $E_u$.  This pulls back to an
  $M'$-linearization $\epsilon_u$ of $q^*E_u$.
  By~\cite[Corollary~4.10]{bosch_lutkeboh91:degenerat_abelian_varieties}, a cubical line bundle
  on $E^\an$ with $M'$-linearization descends to the trivial rigidified line
  bundle on $A$ if and only if it is isomorphic to $(q^*E_u,\epsilon_u)$.
\end{eg}

\begin{notn}
  For a line bundle $L$ on a group scheme $G$ we set
  \[ \sD_2L = m^*L\tensor p_1^*L\inv\tensor p_2^*L\inv, \]
  where $p_1,p_2\colon G\times G\to G$ are the two projections and
  $m\colon G\times G\to G$ is multiplication.
\end{notn}

Let $G = B$ be an abelian variety and let $L$ be a line bundle on $B$.
Recall from~\parref{par:general.notation} that $L$ gives rise to a symmetric
homomorphism $\phi_L\colon B\to B'$ defined on points by
$x\mapsto T_x^*L\tensor L\inv$.  
This is equivalent to the ``universal'' identity $(\Id\times\phi_L)^*P = \sD_2L$. 

The following theorem is the non-Archimedean analogue of Riemann's period relations.

\begin{thm}[{\cite[Theorems~6.7 and~6.13]{bosch_lutkeboh91:degenerat_abelian_varieties}}]
  \label{thm:descend.lbs}
  There is a canonical bijective correspondence between the following data:
  \begin{enumerate}
  \item Rigidified line bundles on $A$.
  \item Equivalence classes of triples $(L,\lambda,c)$, where $L$ is a
    rigidified line bundle on $B$, $\lambda\colon M'\to M$ is a group
    homomorphism, and $c\colon M'\to\Phi^*L$ is a trivialization (as a
    rigidified line bundle on $M'$), satisfying
    \[ \Phi'\circ\lambda = \phi_L\circ\Phi \sptxt{and}
    c(u_1'+u_2')\tensor c(u_1')\inv\tensor c(u_2')\inv = t(u_1',\lambda(u_2'))
    \]
    for all $u_1',u_2'\in M'$, where we use the canonical identification
    $\sD_2L = (\Id\times\phi_L)^* P$.  Two triples $(L_1,\lambda_1,c_1)$ and
    $(L_2,\lambda_2,c_2)$ are equivalent provided that $\lambda_1=\lambda_2$ and
    there exists $u\in M$ such that $L_1\tensor L_2\inv\cong E_u$ and
    $c_1\tensor c_2\inv \cong \epsilon_u$ , where $\epsilon_u$ is defined
    in~\eqref{eq:epsilon_u}.
  \end{enumerate}
  The above correspondence is compatible with tensor product in the obvious way.
  Moreover, the line bundle $L_A$ on $A$ corresponding to $(L,\lambda,c)$ is ample
  if and only if $L$ is ample and the symmetric form
  $(u_1',u_2')\mapsto[u_1',\lambda(u_2')]$ is positive-definite, in
  which case
  \[ \dim H^0(A,\,L_A) = [M:\lambda(M')]\,\dim H^0(B,\,L). \]
\end{thm}

The correspondence is constructed by descent with respect to the subgroup
$M'\subset E^\an$.  The $M'$-linearization on $q^*L$ defined by a triple in~(2)
is described as follows, in terms of
sections~\cite[Proposition~4.9]{bosch_lutkeboh91:degenerat_abelian_varieties}: the isomorphism
$q^*L\isom T_{u'}^*q^*L$ corresponding to the action of $u'\in M'$ on $q^*L$ is
given by $f\mapsto c(u')\tensor e_{\lambda(u')}\tensor f$, up to canonical
isomorphism.  (The canonical isomorphism in question is the pullback via $q$ of
$T_{\Phi(u')}^*L\cong L(\Phi(u'))\tensor E_{\lambda(u')}\tensor L$, where
$L(\Phi(u'))$ is the constant line bundle on the fiber of $L$ over $\Phi(u')$.
See~\cite[\S4]{bosch_lutkeboh91:degenerat_abelian_varieties}.)

\begin{eg}\label{eg:toric.2}
  Suppose that $A$ has toric reduction, as in Example~\ref{eg:toric.1}.  Since
  there are no nontrivial line bundles on $\bT^\an$, by descent theory, the
  Picard group $\Pic(A)$ is naturally isomorphic to
  $H^1(M',\sO(\bT^\an)^\times)$, the group cohomology of $M'$ with coefficients
  in $\sO(\bT^\an)^\times$.  Any analytic invertible function on $\bT^\an$ is in
  fact algebraic~\cite[Theorem~6.3.3(1)]{fresnel_vanderput04:rigid_analytic_geometry}, hence is a
  constant times a character; in other words,
  $\sO(\bT^\an)^\times = K^\times\times M$.  Let $Z\colon u'\mapsto Z_{u'}$ be a
  $1$-cocycle representing a line bundle $L\in\Pic(A)$.  Up to coboundary, we
  can write $Z_{u'} = c(u') \chi^{\lambda(u')}$ for some $c(u')\in K^\times$ and
  $\lambda(u')\in M$, such that $u'\mapsto\lambda(u')$ is a homomorphism from
  $M'$ to $M$, and for $u_1',u_2'\in M'$ one has
  \begin{equation}\label{eq:c.t.relation.toric}
    c(u_1'+u_2') c(u_1')\inv c(u_2')\inv = \chi^{\lambda(u_2')}(u_1') =
    t(u_1',\lambda(u_2'))
  \end{equation}
  and $c(0) = 1$.  A triple corresponding to the line bundle $L$ via
  Theorem~\ref{thm:descend.lbs} is $(\sO,\lambda,c)$.  For $u\in M$, the
  coboundary of $\chi^u\in\sO(\bT^\an)^\times$ is $u'\mapsto t(u',u)$, so that
  $(\sO,\lambda,c)$ and $(\sO,\lambda,t(\scdot,u)c)$ define the same line
  bundle.  Unwrapping the definitions and using Example~\ref{eg:toric.1}, we
  find that the trivialization $\epsilon_u$ of~\eqref{eq:epsilon_u} reduces to
  the homomorphism $t(\scdot,u)$ in this case.
\end{eg}

\subsection{Translations of line bundles}\label{par:trans.lbs}
To $x'\in E'(K)$ we can associate the translation-invariant line bundle $L_{x'}$
on $B$ corresponding to $q'(x')\in B'(K)$: that is,
$L_{x'} = P_{B\times q'(x')}$, where $P\to B\times B'$ is the Poincar\'e bundle.
For $u'\in M'$ we have $E'_{u'} = P_{\Phi(u')\times B'}$, so
$e_{u'}(x')\in P_{\Phi(u')\times q'(x')} = (L_{x'})_{\Phi(u')}$.  It turns out
that $u'\mapsto e_{u'}(x')$ defines a homomorphism $c_{x'}\colon M'\to L_{x'}$
over $\Phi\colon M'\to B$, where we regard $L_{x'}$ as an extension of $B$ by
$\bG_m$.  In fact more is true:

\begin{prop}[{\cite[Corollaries~4.11 and~4.12]{bosch_lutkeboh91:degenerat_abelian_varieties}}]
  \label{prop:invariant.lbs}
  The association $x'\mapsto (L_{x'},c_{x'})$ defines a bijection between points
  of $E'(K)$ and isomorphism classes of pairs $(L,c)$, where $L$ is a
  translation-invariant line bundle on $B$ and $c\colon M'\to L$ is a
  homomorphism lifting $\Phi\colon M'\to B$.  Moreover, $(L_{x'},0,c_{x'})$ is a
  triple as in Theorem~\ref{thm:descend.lbs} giving rise to the
  translation-invariant line bundle on $A$ corresponding to $p'(x')\in A'(K)$.
\end{prop}

See also the paragraph before the statement of Theorem~6.8
in~\cite{bosch_lutkeboh91:degenerat_abelian_varieties}.  When $x' = u\in M\subset E'(K)$ we have
$q'(u) = \Phi'(u)$, so $L_u = E_u$.  Since $e_{u'}(u) = e_u(u')$ it follows that
$c_u = \epsilon_u$ as in Example~\ref{eg:linearize.Eu}, so $(L_u,0,c_u)$ gives
rise to the trivial line bundle on $A$.  This is consistent with
Theorem~\ref{thm:descend.lbs} and the fact that $p'(u) = 0\in A'(K)$.

\begin{prop}\label{prop:triple.for.translate}
  Let $L_A$ be a line bundle on $A$, and let $(L,\lambda,c)$ be a triple
  corresponding to $L_A$ as in Theorem~\ref{thm:descend.lbs}.  Let
  $\phi = \phi_{L_A}\colon A\to A'$ be the homomorphism induced by $L_A$ and let
  $\td\phi\colon E^\an\to E'^\an$ be the lift to universal covers.  Let
  $x\in E(K)$, $y = p(x)\in A(K)$, $z = q(x)\in B(K)$, and
  $x' = \td\phi(x)\in E'(K)$.  Then $T_z^*L\cong L_{x'}\tensor L$, and
  $(T_z^*L, \lambda, c_{x'}\tensor c)$ is a triple corresponding to $T_y^*L_A$.
\end{prop}

\begin{proof}
  As explained in~\parref{par:polarizations} below, we have
  $q'\circ\td\phi = \phi_L\circ q$.  By definition $L_{x'}$ is the line bundle
  defined by $q'(x') = q'(\td\phi(x))$, and $T_z^*L\tensor L\inv$ is the line
  bundle defined by $\phi_L(z) = \phi_L(q(x))$.  Hence
  $L_{x'}\cong T_z^*L\tensor L\inv$.  Similarly, $T_y^*L_A\tensor L_A\inv$ is
  the line bundle defined by $\phi(y)\in A'(K)$; as
  $\phi(y) = p'(\td\phi(x)) = p'(x')$, this line bundle corresponds to the
  triple $(L_{x'},0,c_{x'})$ by Proposition~\ref{prop:invariant.lbs}.  The
  result follows by compatibility of the formation of these triples with tensor
  product.
\end{proof}

\begin{defn}\label{def:translate.triple}
  With the notation in Proposition~\ref{prop:triple.for.translate}, we call the
  triple $(T_z^*L,\lambda,c_{x'}\tensor c)$ the \emph{translate} of
  $(L,\lambda,c)$ by $x\in E(K)$, and we write
  \begin{equation}\label{eq:translate.triple}
    T_x^*(L,\,\lambda,\,c) ~\coloneq~ (T_z^*L,\,\lambda,\,c_{x'}\tensor c).
  \end{equation}
\end{defn}

\begin{eg}\label{eg:toric.2bis}
  Suppose that $A$ has toric reduction, as in Examples~\ref{eg:toric.1}
  and~\ref{eg:toric.2}.  A point $x'\in\bT'(K)$ is equivalent to a homomorphism
  $c_{x'}\colon M'\to K^\times$, and the cocycle $Z\colon u'\mapsto c_{x'}(u')$
  represents a translation-invariant line bundle on $A$, as explained
  in Example~\ref{eg:toric.2}.  This cocycle is a coboundary if and only if
  $x' = u\in M$, in which case $c_u(u') = \chi^{u'}(u) = \chi^u(u') = t(u',u)$.

  If $L_A$ is a line bundle on $A$ with corresponding cocycle
  $Z_{u'} = c(u')\chi^{\lambda(u')}$, then $T_y^*L_A$ corresponds to the cocycle
  $T_{x'}^*Z$. One calculates $(T_{x'}^*Z)_{u'} = c_{x'}(u')c(u')\chi^{\lambda(u')}$,
  which we write as the triple $(\sO,\lambda,c_{x'}\scdot c)$.
\end{eg}

\subsection{Theta functions}\label{par:theta.funcs.mixed}
We will also need an analytic description of \emph{sections} of line bundles on
$A$, building on Theorem~\ref{thm:descend.lbs}.  At this point it is convenient
to pass to invertible sheaves,
following~\cite[\S4]{bosch_lutkeboh91:degenerat_abelian_varieties}.  Our
correspondence between invertible sheaves and line bundles is contravariant: if
$L$ is a line bundle then its corresponding invertible sheaf is
$\sL=\HHom(L,\sO)$.  In particular, a section $s$ of $L$ gives rise to a section
of $\sL\inv$, also denoted $s$.

Let $\sL$ be a rigidified invertible sheaf on $B$.  There is a canonical
decomposition of $H^0(E^\an,q^*\sL)$ as a certain completed direct sum
\[ H^0(E^\an,\,q^*\sL) = \Dsum_{u\in M}^{\wedge} H^0(B,\,\sL\tensor\sE_u)\tensor e_u, \]
where $\sE_u$ is the invertible sheaf on $B$ corresponding to $E_u$ and
$e_u\in H^0(B,\sE_u\inv)$.  In other words, any $f\in H^0(E^\an,\,q^*\sL)$ has a
canonical \emph{Fourier decomposition}
\begin{equation}\label{eq:fourier.decomp}
  f = \sum_{u\in M} a_u\tensor e_u \qquad a_u\in H^0(B,\,\sL\tensor\sE_u).
\end{equation}

\begin{rem}\label{rem:formal.atlas}
  Let $V\subset B^\an$ be a formal affinoid over which the formal Raynaud
  extension~\eqref{eq:raynaud.formal} splits, and which trivializes $\sL$ (as a
  formal line bundle).  Then $q\inv(V)\cong V\times\bT$, and $f|_V$ has a unique
  Laurent series decomposition $f = \sum_{u\in M}a_u\chi^u$ with $a_u\in\sO(V)$.
  The Fourier decomposition~\eqref{eq:fourier.decomp} is the globalization of
  this decomposition.  Moreover, for $x\in E^\an$ one has
  \begin{equation}\label{eq:fourier.norm}
    \|f(x)\|_{q^*\sL} = \left|\sum a_u\chi^u(x)\right|
  \end{equation}
  essentially by the definition of the model metric $\|\scdot\|_\sL$.
  (This also explains the convergence in the completed
  direct sum above.)  See~\cite[\S5]{bosch_lutkeboh91:degenerat_abelian_varieties}.
\end{rem}

\begin{prop}[{\cite[Proposition~5.2]{bosch_lutkeboh91:degenerat_abelian_varieties}}]
  \label{prop:invariant.section}
  Let $\sL_A$ be a rigidified invertible sheaf on $A$, and let $(\sL,\lambda,c)$ be a
  corresponding triple in Theorem~\ref{thm:descend.lbs}.  Let
  $f\in H^0(E^\an,q^*\sL)$ have Fourier decomposition
  $f = \sum_{u\in M} a_u\tensor e_u$.  Then the following are equivalent:
  \begin{enumerate}
  \item $f$ descends to a global section of $\sL_A$;
  \item $f = c(u')\tensor e_{\lambda(u')}\tensor T_{u'}^*f$ for all $u'\in M'$;
  \item for all $(u',u)\in M'\times M$ we have
    \begin{equation}\label{eq:invariant.cond}
      a_{u+\lambda(u')} = t(u',u)\tensor c(u')\tensor T_{\Phi(u')}^*a_u,
    \end{equation}
  \end{enumerate}
 ignoring canonical isomorphisms.
\end{prop}

See the proof of~\cite[Proposition~5.2]{bosch_lutkeboh91:degenerat_abelian_varieties} and the
internal references therein to unwrap the canonical isomorphisms involved.  For
our purposes they are not important; we will only use
Corollary~\ref{cor:fourier.invariant.norm}, which is a consequence.

\begin{defn}\label{defn:theta.function}
  A global section $f\in H^0(E^\an,q^*\sL)$ which descends to a section of
  $\sL_A$ is called a \emph{theta function} with respect to $(\sL,\lambda,c)$.
  We say that $f$ is \emph{invariant} under the linearization action of $M'$.
\end{defn}

We will usually call $f$ simply a ``theta function for $\sL$''.
Note that theta functions for $\sL$ are not intrinsic to $\sL_A$, as
there are multiple triples $(\sL,\lambda,c)$ giving rise to
$\sL_A$.

\begin{cor}\label{cor:fourier.invariant.norm}
  With the notation in Proposition~\ref{prop:invariant.section}, for
  $x\in B^\an$ and $(u',u)\in M'\times M$ we have
  \[ \|a_{u+\lambda(u')}(x)\|_{\sL\tensor\sE_u}
  = \|t(u',u)\|_P\cdot\|c(u')\|_{\sL}\cdot\|a_u(x + \Phi(u'))\|_{\sL\tensor\sE_u} \]
  if $f$ is a theta function.
\end{cor}
\begin{proof}
It follows from Proposition 3.19 that 
\[ a_{u+\lambda(u')} = t(u',u)\tensor c(u')\tensor T_{\Phi(u')}^*a_u, \]
where we have used the canonical isomorphisms 
\[ T^*_{u'}q^*\sL=q^*T^*_{\Phi(u')}\sL=q^*\sL(\Phi(u'))\otimes q^*\sE_{u}\otimes q^*\sL ,\quad T^*_{u'}\sE_{u}=\sE_{u}.\] 
By the construction of the formal metric, it follows that the sections of these line bundles will differ, via this isomorphisms,  by an element of norm $1$. This shows the compatibility of the metrics which proves the claim. 
\end{proof}

\begin{eg}\label{eg:toric.3}
  Suppose that $A$ has toric reduction, as in Examples~\ref{eg:toric.1},
  \ref{eg:toric.2}, and~\ref{eg:toric.2bis}.  Let $L$ be a line bundle on $A$,
  corresponding to a cocycle $u'\mapsto c(u')\chi^{\lambda(u')}$ as in
  Example~\ref{eg:toric.2}.  An analytic function on $\bT$ has the form (Fourier
  decomposition) $f = \sum_{u\in M} a_u\chi^u$ for $a_u\in K$, where
  $\val(a_u)+\angles{u,v}\to\infty$ on the complements of finite subsets of $M$
  for all $v\in N_\R$.  The function $f$ descends to a section of $\sL$ if and
  only if $f = c(u')\chi^{\lambda(u')}T_{u'}^*f$ for all $u'\in M'$, i.e., if
  and only if
  \[ \sum a_{u+\lambda(u')}\,\chi^{u+\lambda(u')}
  = c(u')\chi^{\lambda(u')}\sum a_u\chi^u(u')\,\chi^u
  = \sum t(u',u) c(u') a_u\,\chi^{u+\lambda(u')}. \]
  This recovers~\eqref{eq:invariant.cond} in the totally degenerate case.
\end{eg}

\subsection{Polarizations}\label{par:polarizations}
Let $L_A$ be a rigidified line bundle on $A$, with corresponding triple
$(L,\lambda,c)$ as in Theorem~\ref{thm:descend.lbs}.  Let $\phi_{L_A}\colon A\to A'$
be the homomorphism induced by $L_A$ and let $\td\phi_{L_A}\colon E^\an\to E'^\an$ be
the lift to universal covers, so $\td\phi_{L_A}(M')\subset M$.  It is automatic that
$\td\phi_{L_A}(\bT^\an)\subset\bT'^\an$, so $\td\phi_{L_A}$ induces homomorphisms
$\lambda_{\bT}\colon \bT\to\bT'$ and $\lambda_B\colon B\to B'$.
See~\cite[Proposition~3.5]{bosch_lutkeboh91:degenerat_abelian_varieties}.

\begin{prop}{\cite[Proposition~6.10, Remark~6.11]{bosch_lutkeboh91:degenerat_abelian_varieties}}
  \label{prop:symmetric.phiL}
  In the above situation, the following homomorphisms $M'\to M$ coincide with
  $\lambda$:
  \begin{enumerate}
  \item The homomorphism on character groups induced by $\lambda_\bT$.
  \item The restriction of $\td\phi_{L_A}$ to $M'$.
  \end{enumerate}
  Moreover, $\lambda_B = \phi_L$.
\end{prop}

By Theorem~\ref{thm:descend.lbs}, $\phi_{L_A}$ is a polarization if and only if
$L$ is ample and $[\scdot,\lambda(\scdot)]$ is positive-definite.
In this case the degrees are related by
\[ \deg(\phi_{L_A}) = [M:\lambda(M')]^2\,\deg(\phi_L). \]
See also~\cite[Theorem~6.15]{bosch_lutkeboh91:degenerat_abelian_varieties}.

If $\phi_{L_A}$ is a principal polarization then $\lambda$ is an isomorphism and
$\phi_L$ is also a principal polarization.  In this case there is a unique
nonzero global section $a_0\in H^0(B,L)$ up to scaling.  Taking $u = 0$
in~\eqref{eq:invariant.cond}, the unique theta function $f$ of $L$ (up to
scaling) has Fourier expansion
\begin{equation}\label{eq:riemann.theta}
  f = \sum_{u'\in M'} 1\tensor c(u')\tensor
  T_{\Phi(u')}^*a_0\tensor e_{\lambda(u')}.
\end{equation}
We call $f$ the \emph{Riemann theta function} associated to $(L,\lambda,c)$.  

Recall that for $y\in A(K)$, the line bundle $T_y^*L_A$ defines the same
principal polarization $\phi_{L_A}$.

\begin{prop}\label{prop:val.riemann.theta}
  With the above notation, suppose that $\phi_{L_A}$ is a principal
  polarization.  Then there exists $y\in A(K)$ and a triple $(L,\lambda,c)$ for
  $T_y^*L_A$ such that, if $f = \sum_{u\in M} a_u\tensor e_u$ is the Fourier
  expansion of the Riemann theta function on $L$, then for $u'\in M'$ and
  $x\in B^\an$, we have
  \[ \|a_{\lambda(u')}(x)\|_{L\tensor E_u} =
  \|t(u',\lambda(u'))\|_P^{1/2}\cdot\|a_0(x+\Phi(u'))\|_{L\tensor E_u}. \]
  Moreover, we can choose $y$ and $(L,\lambda,c)$ such that
  \[ \big((-1)^*L,\,\lambda,\,(-1)^*c\big) ~=~
  \big(L,\,\lambda,\,c\big), \]
  where $(-1)\colon B\to B$ is negation and $(-1)^*c(u') = c(-u')$.
\end{prop}

\begin{proof}
  By Corollary~\ref{cor:fourier.invariant.norm}, for the first assertion we only
  need to show that there exist $y$ and $(L,\lambda,c)$ such that
  \begin{equation}\label{eq:c.vs.t}
    \|c(u')\|_L = \|t(u',\lambda(u'))\|_P^{1/2}
  \end{equation}
  for all $u'\in M'$.  Let $(L',\lambda,c')$ be a triple associated with $L_A$.
  Let $Q'(u') = -\log\|c'(u')\|_{L'}$ and let
  $\beta(u_1',u_2') = [u_1',\lambda(u_2')] = -\log\|t(u_1',\lambda(u_2'))\|_P$.
  Then $\beta\colon M'\times M'\to\R$ is a symmetric, positive-definite bilinear
  form by Theorem~\ref{thm:descend.lbs} (since $L_A$ is ample), and for
  $u_1',u_2'\in M'$ we have $\beta(u_1',u_2') = Q'(u_1'+u_2')-Q'(u_1')-Q'(u_2')$
  (also by Theorem~\ref{thm:descend.lbs}).  Letting $Q(u) = \frac 12 \beta(u,u)$ be
  the quadratic form associated to the bilinear form $\beta$, we also have
  $\beta(u_1',u_2') = Q(u_1'+u_2')-Q(u_1')-Q(u_2')$, so $Q-Q'\colon M'\to\R$ is a
  homomorphism.  Since $c(u')$ and $t(u',\lambda(u'))$ are $K$-valued points,
  the image of $Q-Q'$ is contained in $\Gamma\coloneq\val(K^\times)$.  Therefore
  we may regard $Q-Q'$ as a point in
  $N'_\Gamma\coloneq\Hom(M',\Gamma)\subset N'_\R$.

  Choose $x'\in E'(K)$ such that $\trop(x') = Q-Q'$.  Let $L_{x'}$ be the
  translation-invariant line bundle on $\beta$ and $c_{x'}\colon M'\to L_{x'}$ the
  homomorphism which correspond to $x'$ by Proposition~\ref{prop:invariant.lbs}.
  For $u'\in M'$ we have $c_{x'}(u') = e_{u'}(x')$, so
  \[ \angles{u',\trop(x')} = -\log\|c_{x'}(u')\|_{E_{u'}} \]
  by definition of $\trop$~\eqref{eq:trop.E}.  On the other hand,
  \[ \|c_{x'}(u')\|_{E_{u'}} = \|c_{x'}(u')\|_P = \|c_{x'}(u')\|_{L_{x'}} \]
  since any translation-invariant line bundle is a pullback of the Poincar\'e
  bundle as a formal line bundle.  It follows that
  $Q-Q' = -\log\|c_{x'}(\scdot)\|_{L_{x'}}\in N'_\Gamma$.

  Let $\td\phi_{L_A}\colon E^\an\isom E'^\an$ be the lift of $\phi_{L_A}$ to
  universal covers, let $x = \td\phi_{L_A}\inv(x')$, let $y = p(x)\in A(K)$, and
  let $z = q(x)\in B(K)$.  Then
  $(L,\lambda,c)\coloneq T_x^*(L',\lambda',c') = (T_z^*L',\lambda,c_{x'}\tensor
  c')$
  is a triple for $T_y^*L_A$ by Proposition~\ref{prop:triple.for.translate}, and
  by construction, $-\log\|c(u')\|_L = Q(u') = \frac 12 \beta(u',u')$.

  Now we treat the final assertion.  To begin we replace $L_A$ by $T_y^*L_A$ and
  $(L',\lambda,c')$ by $(L,\lambda,c)$. A triple for the translation-invariant
  line bundle $(-1)^*L_A\tensor L_A\inv$ is
  \[ \big((-1)^*L\tensor L\inv,\, 0,\, (-1)^*c\tensor c\inv \big). \]
  For $u'\in M'$ we have $\|c(u')\|_L = \|c(-u')\|_L$ by~\eqref{eq:c.vs.t}, so
  $\|(-1)^*c(u')\tensor c(u')\inv\|_{(-1)^*L\tensor L\inv} = 1$.  Let
  $x'\in E'(K)$ be the point giving rise to the pair
  $((-1)^*L\tensor L\inv,(-1)^*c\tensor c\inv)$ in the manner of
  Proposition~\ref{prop:invariant.lbs}, and choose $x''\in E'(K)$ such that
  $2x'' = x'$.  Then $L_{x''}^{\tensor 2}\cong(-1)^*L\tensor L\inv$ and
  $c_{x''}^{\tensor 2}\cong(-1)^*c\tensor c\inv$.  Since
  $(-1)^*L_{x''} = L_{x''}\inv$ as $L_{x''}$ is translation-invariant, we have
  \[ (-1)^*(L\tensor L_{x''}) \cong (-1)^*L\tensor L_{x''}\inv \cong L\tensor
  L_{x''}, \] and similarly for $c\tensor c_{x''}$.  Hence the triple
  \[ \big(L\tensor L_{x''},\, \lambda,\, c\tensor c_{x''} \big) =
  T_{x''}^*(L,\,\lambda,\,c) \]
  is invariant under the symmetry $(-1)$, and for $u'\in M'$,
  \[ \|c_{x''}(u)\|_{L_{x''}} =
  \sqrt{\|(-1)^*c(u')\tensor c(u')\inv\|_{(-1)^*L\tensor L\inv}} = 1. \]
  It follows that $T_{x''}^*(L,\lambda,c)$ satisfies all of the hypotheses of
  the Proposition.
\end{proof}

\begin{eg}\label{eg:toric.4}
  Suppose that $A$ has toric reduction, as in Examples~\ref{eg:toric.1},
  \ref{eg:toric.2}, \ref{eg:toric.2bis}, and~\ref{eg:toric.3}.  Let $L_A$ be a
  line bundle on $A$, with associated cocycle
  $u'\mapsto c(u')\chi^{\lambda(u')}$.  Let $\lambda_\bT\colon\bT\to\bT'$ be the
  homomorphism inducing $\lambda\colon M'\to M$ on character groups.
  By~\eqref{eq:c.t.relation.toric} the bilinear form
  $(u_1',u_2')\mapsto t(u_1',\lambda(u_2')) = \chi^{\lambda(u_2')}(u_1')$ is
  symmetric.  Hence for $u_1',u_2'\in M'\times M'$, we have
  \begin{equation}\label{eq:compat.lambdas}
    \chi^{u_2'}(\lambda_\bT(u_1')) = \chi^{\lambda(u_2')}(u_1')
    = t(u_1',\lambda(u_2')) = t(u_2',\lambda(u_1'))
    = \chi^{u_2'}(\lambda(u_1')),
  \end{equation}
  where the first equality is the definition of $\lambda_\bT$ and the last is
  the definition of the embedding $M\inject\bT'$ in Example~\ref{eg:toric.1}.
  Hence $\lambda$ is the restriction of $\lambda_\bT$ to $M'$, as in
  Proposition~\ref{prop:symmetric.phiL}. 

  Suppose that $\phi_{L_A}$ is a principal polarization.  Choose a bilinear form
  $\sqrt{t(\scdot,\lambda(\scdot))}\colon M'\times M'\to K^\times$ whose square
  is $t(\scdot,\lambda(\scdot))$ (choose square roots of the images of pairs of
  basis elements).  Then
  \[ f = \sum_{u'\in M'} \sqrt{t(u',\lambda(u'))}\,\chi^{\lambda(u')} \]
  is the Riemann theta function associated to the triple $(\sO,\lambda,c)$ for
  $c(u') = \sqrt{t(u',\lambda(u'))}$. See also~\cite[Definitions~16,19]{teitelbaum88:genus_two}.
\end{eg}

%%%%%%%%%%%%%%%%%%%%%%%%%%%%%%%%%%%

\vskip 1cm

\section{Tropicalization of theta functions}
\label{sec:tropicalization.theta}

An abelian variety $A$ over $K$ has a canonical skeleton $\Sigma$, which is a
tropical abelian variety in the sense of Definition~\ref{def:trop.av}.  In this
section we define the tropicalization $f_{\trop}$ of a theta function $f$ on
$A$, and we prove that $f_{\trop}$ is a tropical theta function on $\Sigma$.  We
show that a Riemann theta function tropicalizes to the tropical Riemann theta
function in the principally polarized case, up to translation and scaling.  

Throughout this section we use the notation~\parref{par:unif.dual} for the
uniformization of $A$ and its dual.

\subsection{The skeleton}\label{par:skeleton}
There is a canonical continuous section $\sigma\colon N_\R\to E^\an$ of the
tropicalization map $\trop\colon E^\an\to N_\R$ defined in~\eqref{eq:trop.E},
which is constructed in~\cite[Example~7.2]{gubler10:canonical_measures} as
follows.  Choose a local section $V\to E^\an$ of the formal Raynaud
extension~\eqref{eq:raynaud.formal} as in Remark~\ref{rem:formal.raynaud}, so
$q\inv(V)\cong\bT\times V$.  Then $\trop\inv(v)\cong U_v\times V$, where $U_v$
is the inverse image of $v$ under $\trop\colon\bT^\an\to N_\R$.  Any analytic
function $h$ on $U_v\times V$ has a unique Laurent series expansion
\[ h = \sum_{u\in M} a_u\chi^u \sptxt{such that}
|a_u|_{\sup}\,\exp(-\angles{u,v})\To 0 \]
on the complements of finite sets of $M$, where $a_u\in\sO(V)$ and
$|\scdot|_{\sup}$ is the supremum norm on $\sO(V)$.  We define
$\sigma(v)$ to be the norm on $\sO(U_v\times V)$ given by
\begin{equation}\label{eq:sigma.v}
  \big|h(\sigma(v))\big| = \max_{u\in M}\big\{ |a_u|_{\sup}\,\exp(-\angles{u,v})\big\}.
\end{equation}
This is independent of the choice of $V$ and the choice of splitting (of the
\emph{formal} Raynaud extension), and the resulting section
$\sigma\colon N_\R\to E^\an$ is continuous.  

\begin{lem}\label{lem:skel.to.gauss}
  Let $\fB$ be the smooth abelian $R$-scheme with generic fiber $B$.
  There is a unique point $\xi\in B^\an$ whose reduction is the generic point of
  $\fB_s$, and $q(\sigma(v)) = \xi$ for all $v\in N_\R$.
\end{lem}

\begin{proof}
  The first assertion is a standard application
  of~\cite[Proposition~2.4.4(ii)]{berkovic90:analytic_geometry}: the point $\xi$
  corresponds to the supremum norm on $\sO(V)$ for any $V$ as
  in~\parref{par:skeleton}.  If $h = a_0$ is the pullback of a function on $V$
  then by definition $|h(\sigma(v))| = |a_0|_{\sup}$, so $q(\sigma(v)) = \xi$.
\end{proof}

\begin{lem}\label{lem:translate.skel}
  Let $v \in N_\R$ and $x\in E(K)$.  Then $\sigma(v) + x \coloneq T_x(\sigma(v))$ is
  equal to $\sigma(v + \trop(x))$.
\end{lem}

\begin{proof}
  Let $w = \trop(x)$, choose $V$ as in~\parref{par:skeleton}, and let $h$ be an
  analytic function on $\bT^\an\times V$.  We can write
  $h = \sum_{u\in M} a_u\chi^u$ for $a_u\in\sO(V)$, such that
  $|a_u|_{\sup}\exp(-\angles{v',u})\to 0$ for all $v'\in N_\R$.  We have
  \[\begin{split}
      |h(T_x(\sigma(v)))| 
      &= \left|\sum a_u T_x^*(\chi^u)(\sigma(v))\right| \\
      &= \max\big\{ |a_u|_{\sup}\,\exp(-\angles{u,v}-\angles{u,w}) \big\} \\
      &= |h(\sigma(v+w))|
\end{split}\]
  because $T_x^*(\chi^u) = \chi^u(x)\chi^u$ and
  $|\chi^u(x)| = \exp(-\angles{u,\trop(x)})$ by definition of $\trop$.
\end{proof}

\begin{defn}\label{def:skel.av}
  The \emph{skeleton} of $A$ is the real torus $N_\R/\trop(M')$.
\end{defn}

Note that the pairing $[\scdot,\scdot]\colon M'\times M\to\R$
of~\eqref{eq:pairing.1} defines the embedding $\trop\colon M'\to N_\R$ used in
Definition~\ref{def:skel.av}.
The tropicalization map $\trop\colon E^\an\to N_\R$ defines by passage to the
quotient a proper, surjective homomorphism $\tau\colon A^\an\to\Sigma$, and the
section $\sigma$ restricts fiberwise to a section $\sigma\colon\Sigma\to A^\an$.
The composition $\sigma\circ\tau\colon A^\an\to A^\an$ is the image of a
deformation retraction onto $\sigma(\Sigma)$
by~\cite[\S6.5]{berkovic90:analytic_geometry}.

\begin{eg}\label{eg:toric.5}
  Suppose that $A$ has toric reduction, as in Examples~\ref{eg:toric.1},
  \ref{eg:toric.2}, \ref{eg:toric.2bis}, \ref{eg:toric.3}, and~\ref{eg:toric.4}.  Then
  $E^\an=\bT^\an$, and for $v\in N_\R$ the point $\sigma(v)$ is defined by
  \[ |f(\sigma(v))| = \sup_{u\in M} |a_u|\,\exp(-\angles{u,v}) \sptxt{for}
  f = \sum_{u\in M} a_u\,\chi^u. \]
  In this case $a_u\in K$ are constants.  The map $\sigma\colon N_\R\to\bT^\an$
  is the usual skeleton of a torus.
\end{eg}

\subsection{Skeletons and polarizations}\label{par:skel.polarization}
Suppose now that $A$ is endowed with a polarization $\phi\colon A\to A'$
associated to an ample line bundle $L_A$.  By~\parref{par:polarizations}, $L_A$
gives rise to a triple $(L,\lambda,c)$ such that the real-valued pairing
\[ [\scdot,\,\lambda(\scdot)]\colon~ M'\times M'\To \R \]
is bilinear, symmetric, and positive-definite.  The dual abelian variety $A'$
has skeleton $\Sigma' \coloneq N_\R'/\trop(M)$, where $N_\R' = \Hom(M',\R)$ and
$\trop\colon M\to N_\R'$ is again defined by $[\scdot,\scdot]$.
From now on we identify $M'$ and $M$ with their images in $N_\R$ and $N_\R'$,
respectively.  
\begin{prop}\label{prop:skel.polarization}
  The pairing $[\scdot,\scdot]\colon M'\times M\to\R$ and the homomorphism
  $\lambda\colon M'\to M$ make $\Sigma = N_\R/M'$ into a tropical abelian
  variety in the sense of Definition~\ref{def:trop.av}, with dual
  $\Sigma' = N_\R'/M$.  Moreover, the square
  \[\xymatrix @R=.2in{
    {A^\an} \ar[r]^{\phi} \ar[d]_{\tau} &
    {A'^\an} \ar[d]^{\tau} \\
    {\Sigma} \ar[r]_{\phi_{\trop}} & {\Sigma'} }\]
  is commutative, where $\phi_{\trop}$ is the polarization defined by $\lambda$.
\end{prop}
\begin{proof}
The homomorphism $\lambda\colon M'\to M$ gives rise to a map $\td\phi_{\trop}\colon N_\R\to N_\R'$ making the square
\[\xymatrix @R=.2in{
  {E^\an} \ar[r]^{\td\phi} \ar[d]_{\trop} &
  {E'^\an} \ar[d]^{\trop} \\
  {N_\R} \ar[r]_{\td\phi_{\trop}} & {N_\R'} }\]
commutative.  Since $\td\phi$ restricts to $\lambda\colon M'\to M$ on
sublattices by Proposition~\ref{prop:symmetric.phiL}, $\td\phi_{\trop}$ also
restricts to $\lambda$ on $M'\subset N_\R$, and hence descends to a homomorphism
$\phi_{\trop}\colon\Sigma\to\Sigma'$.  Unwinding the definitions, for
$u_1',u_2'\in M'$ we have
\[ \angles{u_1',\,\lambda(u_2')} = -\log\|t(u_1',\,\lambda(u_2'))\|_P 
= [u_1',\,\lambda(u_2')], \]
(the first being the evaluation pairing $M'\times N'_\R\to\R$),
which is analogous to~\eqref{eq:compat.lambdas}. 
\end{proof}

See also~\cite[(3.7) and~\S4]{baker_rabinoff14:skeleton_jacobian} for a
discussion in the principally polarized case.

\subsection{Tropicalization of theta functions}\label{par:tropicalization.theta}
Fix a rigidified line bundle $L_A$ on $A$, let $(L,\lambda,c)$ be a corresponding
triple of Theorem~\ref{thm:descend.lbs}, and let $\phi_{L_A}\colon A\to A'$
be the induced homomorphism.  Let $f\in H^0(E^\an,q^*L)$ be a nonzero theta
function (Definition~\ref{defn:theta.function}).

\begin{defn}\label{def:tropicalize.theta}
  The \emph{tropicalization} of $f$ is the function
  \[ f_{\trop}\colon N_\R\to\R \sptxt{defined by}
  f_{\trop}(v) = -\log\|f(\sigma(v))\|_{q^*L}.
  \]
\end{defn}

Here $\|\scdot\|_{q^*L}$ is the pullback of the canonical model metric
$\|\scdot\|_L$ of~\parref{par:model.metrics}.  The next proposition shows that $\|f(\sigma(v))\|_{q^*L}\neq 0$ for all
$v\in N_\R$.

\begin{thm}[The tropicalization of a theta function is a tropical theta function]\label{thm:trop.theta}
  Let $f\in H^0(E^\an,q^*L)$ be a nonzero theta function with tropicalization
  $f_{\trop}\colon N_\R\to\R$ and Fourier expansion
  $f = \sum_{u\in M} a_u\tensor e_u$.  Let $\xi\in B^\an$ be the point defined in
  Lemma~\ref{lem:skel.to.gauss}. 
  \begin{enumerate}
  \item $\displaystyle f_{\trop}(v) 
    = \min_{u\in M} \big\{-\log\|a_u(\xi)\|_{L\tensor E_u} 
    + \angles{u,v}\big\} < \infty$
    for all $v\in N_\R$.
  \item $f_{\trop}$ is a piecewise integral affine function.
  \item Suppose that $L_A$ is ample, so that $\phi_{L_A}$ is a polarization.
    Let $c_{\trop}(u') = -\log\|c(u')\|_L$.  Then
    \[ c_{\trop}(u_1'+u_2')-c_{\trop}(u_1')-c_{\trop}(u_2') =
    [u_1',\,\lambda(u_2')] \sptxt{for all $u_1',u_2'\in M'$,} \]
    and for $u'\in M'$ and $v\in N_\R$ we have
    \begin{equation}\label{eq:trop.theta.trans.law}
      f_{\trop}(v) = f_{\trop}(v + u') + c_{\trop}(u') + \angles{\lambda(u'),v}.
    \end{equation}
  \end{enumerate}
  In particular, $f_{\trop}$ is a tropical theta function in the sense of
  Definition~\ref{def:trop.riemann.theta}. 
\end{thm} 

\begin{proof}
  Choose a local splitting $q\inv(V)\cong\bT\times V$ of the formal Raynaud
  extension~\eqref{eq:raynaud.formal} which also trivializes $L$, as in
  Remark~\ref{rem:formal.raynaud}.  Note that $\Sigma\subset V$ by
  Lemma~\ref{lem:skel.to.gauss}.  On $q\inv(V)$ we can write
  \[ f = \sum_{u\in M} a_u\chi^u \sptxt{where} a_u\in\sO(V), \]
  and some $a_u$ is nonzero.  We have
  \[ \|f\|_{q^*L} = \left|\sum a_u\chi^u(\sigma(v))\right| = \max_{u\in M}\big\{
  |a_u|_{\sup}\,\exp(-\angles{u,v})\big\} = \max_{u\in M}\big\{
  |a_u(\xi)|\,\exp(-\angles{u,v})\big\}.
  \]
  by~\eqref{eq:fourier.norm}, \eqref{eq:sigma.v}, and
  Lemma~\ref{lem:skel.to.gauss}.  This proves~(1).  Clearly $f$ is piecewise
  affine, and the slopes are integral because the linear part of $f$ on any
  domain of linearity has the form $\angles{u,\scdot}$, which takes integer
  values on $N$ (note $|a_u(\xi)|$ is constant with respect to $v$); this
  proves~(2).

The identity for $c_{\trop}(\cdot)$ is obtained by applying logarithm to the identity for $c(\cdot)$ in Theorem~\ref{thm:descend.lbs}.
  By Proposition~\ref{prop:invariant.section} we have
  $f = c(u')\tensor e_{\lambda(u')}\tensor T_u^*f$ for all $u'\in M'$.
  Therefore
  \[ \|f(\sigma(v))\|_{q^*L} =
  \|c(u')\|_L\cdot\|e_{\lambda(u')}(\sigma(v))\|_{E_{\lambda(u')}}
  \cdot\|f(\sigma(v)+u')\|_{q^*L}. \] We have
  \[ -\log\|e_{\lambda(u')}(\sigma(v))\|_{E_{\lambda(u')}} =
  \angles{\lambda(u'),\,\trop(\sigma(v))} = \angles{\lambda(u'),\,v} \]
  by definition of $\trop$~\eqref{eq:trop.E}.  Since
  $\sigma(v) + u' = \sigma(v + u')$ by Lemma~\ref{lem:translate.skel}, this
  completes the proof.
\end{proof}

\subsection{Theta functions and translation}\label{par:translate.theta}
Before discussing the Riemann theta function, we mention how tropicalization of
theta functions behaves with respect to translation.  This is somewhat
complicated by the fact that theta functions are only defined in the presence of
a triple $(L,\lambda,c)$, which must also be translated.  Let $L_A$ be a
rigidified line bundle on $A$. By Theorem \ref{thm:descend.lbs}, it has an associated triple $(L,\lambda,c)$, where $L$ is a rigidified line bundle on $B$. Recall that we have an algebraic morphism $q:E\to B$ and an analytic morphism $p:E^{\text{an}}\to A^{\text{an}}$ appearing in the exact sequences \eqref{raynaud1} and \eqref{raynaud2}, respectively. Fix a nonzero theta function $f\in H^0(E^\an,q^*L)$ and a point $x\in E(K)$. We wish to explain how translation of $f$ by $x$ interacts with tropicalization. To this end, denote the images of $x$ under $p$ and $q$ by $y = p(x)\in A(K)$ and $z = q(x)\in B(K)$, respectively.  Let $\phi_{L_A}\colon A\to A'$ be
the homomorphism induced by $L_A$, as described in \ref{par:general.notation}, and let $\td\phi_{L_A}\colon E\to E'$ be its lift
to the universal covers $E$ of $A$ and $E'$ of $A'$. Define $x'\in E'(K)$ to be the image of $x$ under $\phi_{L_A}$.  By
Proposition~\ref{prop:triple.for.translate}, $(T_z^*L,\lambda,c_{x'}\tensor c)$
is a triple for $T_y^*L_A$.  It is clear that we have canonical isomorphisms of line bundles
	$$
	T_x^*q^*L\ \cong\ q^*T_z^*L
	\ \ \ \ \mbox{on}\ \ 
	E
	$$
and
	$$
	T_x^*p^*L_A\ \cong\ p^*T_y^*L_A
	\ \ \ \ \mbox{on}\ \ \ 
	E^{\text{an}}.
	$$
Since $f$ is a theta function, there exists a
unique section $f_A\in H^0(A,L_A)$ such that $f = p^*f_A$. Thus
$T_x^*f = p^*T_y^*f_A\in H^0(E^\an,q^*T_z^*L)$ is a theta function with respect
to the translated triple $(T_z^*L,\lambda,c_{x'}\tensor c) = T_x^*(L,\lambda,c)$.

\begin{lem}\label{lem:translate.theta}
  With the above notation, $(T_x^*f)_{\trop} = T_{\trop(x)}^*f_{\trop}$ for
  all $x\in E(K)$, where for any $v\in N_{\bold{R}}$, we define
  \[ (T_{v}^*f_{\trop})(w) \coloneq f_{\trop}(v+w). \]
\end{lem}

\begin{proof}
  The canonical metric on $T_z^*L$ coincides with the pullback via $T_z$ of
  $\|\scdot\|_L$ since translation by $z$ extends to the special fiber of the
  smooth model of $B$.  The Lemma follows from this and
  Lemma~\ref{lem:translate.skel}.
\end{proof}

\begin{thm}[The tropicalization of a Riemann theta function is a tropical Riemann theta function]\label{thm:trop.riemann.theta}
  Let $L_A$ be an ample rigidified line bundle on $A$ defining a principal
  polarization.  Let $(L,\lambda,c)$ be a triple corresponding to $L_A$ and let
  $f\in H^0(E^\an,q^*L)$ be the Riemann theta function for $L$, as defined
  in~\eqref{eq:riemann.theta}.  Then there exists $x\in E(K)$ such that
  $(T_x^*f)_{\trop} = T_{\trop(x)}^*f_{\trop}$ is equal to the tropical Riemann
  theta function associated to $\lambda$, up to an additive constant.
\end{thm}

\begin{proof}
  By Proposition~\ref{prop:val.riemann.theta},
  Theorem~\ref{thm:trop.theta}(1), and Lemma~\ref{lem:skel.to.gauss}, there
  is a translate of $(L,\lambda,c)$ whose Riemann theta function tropicalizes to
  the tropical Riemann theta function (up to an additive constant).  The
  translate of $f$ must be the Riemann theta function of this translate, so the
  Corollary follows from Lemma~\ref{lem:translate.theta}.
\end{proof}

\section{Application to constructing rational functions}
\label{sec:constr-ratl-func}

We apply the results of~\S\ref{sec:tropicalization.theta} to give a method for
tropically constructing rational functions $f$ on $A$ such that $-\log|f|$ has a
prescribed behavior on the skeleton $\Sigma$.  First we remark that if $f$ is a
nonzero rational function on $A$, then $|f(x)|\neq 0$ for all $x\in\Sigma$: this
follows from Theorem~\ref{thm:trop.theta} as applied to the trivial line bundle.

\begin{defn}\label{def:trop.rational}
  Let $f$ be a nonzero rational function on $A$.  Its \emph{tropicalization} is
  the function
  \[ f_{\trop}\colon\Sigma\To\R \sptxt{defined by}
  f_{\trop}(v) = -\log|f(\sigma(v))|.
  \]
\end{defn}

The composition $f\circ p\colon E^\an\to\R$ is a theta function with respect to
the trivial line bundle, and the composition of $f_{\trop}$ with the quotient
map $N_\R\to\Sigma$ coincides with $(f\circ p)_{\trop}$
(Definition~\ref{def:tropicalize.theta}).  It follows from
Theorem~\ref{thm:trop.theta} that $f_{\trop}$ is piecewise integral affine
(Definition~\ref{defn:unimodularity}).  
In what follows we will identify
functions on $\Sigma$ with functions on $N_\R$ which are invariant under the
translation action of $M'$.

\begin{thm}\label{thm:same.automorphy}
  For $i=1,2$ let $L_{i,A}$ be an ample line bundle on $A$, let
  $(L_i,\lambda_i,c_i)$ be a triple corresponding to $L_{i,A}$ via
  Theorem~\ref{thm:descend.lbs}, and let $f_i\in H^0(E^\an,q^*L_i)$ be a nonzero
  theta function with tropicalization $f_{i,\trop}$.  Suppose that
  $f_{1,\trop}-f_{2,\trop}$ is invariant under translation by elements of $M'$,
  i.e., that $f_{1,\trop}$ and $f_{2,\trop}$ satisfy the same transformation
  law~\eqref{eq:trop.theta.trans.law}.  Then there exists a rational function
  $h$ on $A$ such that $h_{\trop} = f_{1,\trop}-f_{2,\trop}$.
\end{thm}

\begin{proof}
  Let $f = f_1/f_2$, regarded as a section of $q^*(L_1\tensor L_2\inv)$.  A
  triple corresponding to $L \coloneq L_{1,A}\tensor L_{2,A}\inv$ is
  $(L,\lambda,c) \coloneq (L_1\tensor L_2\inv,\lambda_1-\lambda_2,c_1\tensor
  c_2\inv)$.
  Clearly $\lambda_i$ is determined by~\eqref{eq:trop.theta.trans.law} (fix $u'$
  and vary $v$), so $\lambda_1 = \lambda_2$ and $\lambda=0$.  The Fourier
  expansion for $f$ has the form $f = \sum_{u\in M} a_u\tensor e_u$, and
  satisfies the invariance property~\eqref{eq:invariant.cond}, which in this
  case says
  \[ a_u = t(u',u)\tensor c(u')\tensor T_{\Phi(u')}^*a_u \]
  for all $(u', u)\in M'\times M$.  Define
  $f' = a_0\tensor e_0\in H^0(E^\an,q^*L)$.  Then $f'$ is also a theta function,
  and $f'_{\trop}$ is constant by Theorem~\ref{thm:trop.theta}(1).
  Multiplying $f'$ by a scalar, we may assume $f'_{\trop}$ is identically zero.
  Then $h = f/f'$ is a rational function on $A$ such that
  $h_{\trop} = f_{\trop}$.
\end{proof}

\subsection{Example constructions}
Here we give concrete examples of Theorem~\ref{thm:same.automorphy}.  These are translations of standard constructions into this tropical/non-Archimedean setting.

\begin{lem}\label{lem:difference.thetas}
  Let $L_A$ be an ample rigidified line bundle on $A$ with associated triple
  $(L,\lambda,c)$.  Choose points $x_1,\ldots,x_n\in E(K)$ such that
  $\sum_{i=1}^n x_i = 0$.  Then $\Tensor_{i=1}^nT_{x_i}^*(L,\lambda,c)$ is
  canonically isomorphic to $(L^{\tensor n}, n\lambda, c^{\tensor n})$.
\end{lem}

\begin{proof}
  Let $\phi = \phi_{L_A}\colon A\to A'$, let $\td\phi\colon E^\an\to E'^\an$ be
  the lift to universal covers, and let $y_i = p(x_i)\in A(K)$ and
  $x_i' = \td\phi(x_i)\in E'(K)$.  Then
  $T_{x_i}^*(L,\lambda,c) = (L_{x_i'}\tensor L,\lambda,c_{x_i'}\tensor c)$,
  where $L_{x_i'}$ and $c_{x_i'}$ are defined in~\parref{par:trans.lbs}.  We
  have $\Tensor_{i=1}^n L_{x_i'} = L_{\sum x_i'} = \sO$ since $L_{x_i'}$ is the
  translation-invariant line bundle on $B$ corresponding to $q'(x_i')\in B'(K)$.
  Moreover, for fixed $u'\in M'$ we have
  $c_{x_i'}(u') = e_{u'}(x_i')\in P_{\Phi(u')\times B'} = E'_{u'}$; since
  $e_{u'}$ is a homomorphism, $\sum_{i=1}^n c_{x_i'}(u') = 0$.
\end{proof}

\begin{defn}\label{def:level.n}
  Fix an ample line bundle $L_A$ giving rise to a principal polarization
  $\phi_{L_A}\colon A\to A'$, and fix a triple $(L,\lambda,c)$ corresponding to
  $L_A$.  A \emph{level-$n$ theta function} with respect to $(L,\lambda,c)$ is a
  theta function associated to $(L^{\tensor n},n\lambda,c^{\tensor n})$.
\end{defn}

\begin{prop}\label{prop:construct.level.n}
  Fix an ample line bundle $L_A$ giving rise to a principal polarization
  $\phi_{L_A}\colon A\to A'$, and fix a triple $(L,\lambda,c)$ corresponding to
  $L_A$.  Let $f\in H^0(E^\an,q^*L)$ be the Riemann theta function for $L$.
  Choose points $x_1,\ldots,x_n\in E(K)$ such that $\sum_{i=1}^n x_i = 0$.  Then
  $\Tensor_{i=1}^n T_{x_i}^*f$ is a level-$n$ theta function.  In particular,
  $f^{\tensor n}$ is a level-$n$ theta function.
\end{prop}

This is immediate from Lemma~\ref{lem:difference.thetas} and the discussion
in~\eqref{par:translate.theta}.  It follows that
$h = \Tensor_{i=1}^n T_{x_i}^*f/f^{\tensor n}$ is invariant under $M'$ and is
therefore a rational function on $A$.  By Lem\-ma~\ref{lem:translate.theta}, for
all $v\in N_\R$,
\begin{equation}\label{eq:difference.thetas}
  h_{\trop}(v) = \sum_{i=1}^n f_{\trop}(v + v_i) - nf_{\trop}(v),
\end{equation}
where $v_i = \trop(x_i)$.  In particular, there exists a rational function on $A$ whose tropicalization is the right side of~\eqref{eq:difference.thetas}, where $f_{\trop}$ is the tropical Riemann theta function on $\Sigma$.

\subsection{Kummer functions}\label{par:kummer.property}
Certain level-$2$ theta functions are symmetric with respect to negation.
Thus~\eqref{eq:difference.thetas} will often produce tropically Kummer
functions, in the sense of Definition~\ref{def:tropically.kummer}.  We wish to
prove the other direction: given a tropically Kummer function, we want to
construct a rational map which is symmetric with respect to negation.

Let $A/(-1)$ be the quotient in which we identify each
pair of points $x$ and $-x$ in $A$, and define $\Sigma/(-1)$ similarly. The
former quotient has the structure of an algebraic variety, called the
{\em Kummer variety} of $A$. The quotienting morphism $A\surject A/(-1)$ is
called the {\em Kummer map}.
See~\cite[\S4.8]{birkenhak_lange04:complex_abelian_varieties}.

Let $A' = A/(-1)$.  The Kummer map induces an injection
$\Sigma/(-1)\inject A'^\an$, and since the deformation retraction
$A^\an\to\Sigma$ is canonical, it induces a deformation retraction
$A'^\an\to\Sigma/(-1)$.  We call $\Sigma/(-1)$ the \emph{skeleton} of the Kummer
variety $A'$.

\begin{defn}\label{def:kummer.property}
  Let $A$ be an abelian variety over $K$ with skeleton $\Sigma$, and let
  $f\colon A\dashrightarrow\bG_m^{r}$ be a rational map.  We say that $f$ is
  \emph{Kummer} if it factors through the Kummer map $A\surject A/(-1)$.  
\end{defn}

Definition~\ref{def:kummer.property} applies in particular to nonzero rational
functions.  If $f$ is a rational map which is Kummer then $f_{\trop}$ is Kummer
in the sense of Definition~\ref{def:tropically.kummer}: that is, it factors through
the quotient $\Sigma\to\Sigma/(-1)$, hence defines a function
$f_{\trop}\colon\Sigma/(-1)\to\R^r$.

\begin{lem}\label{lem:level.2.kummer}
  Fix an ample line bundle $L_A$ giving rise to a principal polarization
  $\phi_{L_A}\colon A\to A'$, and fix a triple $(L,\lambda,c)$ corresponding to
  $L_A$.  Replacing $L_A$ by a translate, we assume that $L_A$ and
  $(L,\lambda,c)$ satisfy the conclusions of
  Proposition~\ref{prop:val.riemann.theta}. Let $f\in H^0(E^\an,q^*L)$ be the
  Riemann theta function for $L$, and choose $x\in E(K)$.  Then the rational
  function $h = T_x^*f\tensor T_{-x}^*f\tensor f^{\tensor -2}$ of
  Proposition~\ref{prop:construct.level.n} is Kummer.
\end{lem}

\begin{proof}
  By Proposition~\ref{prop:val.riemann.theta}, $(-1)^*L = L$ and $(-1)^*c = c$,
  so $(-1)^*f$ is also a theta function for $L$.  Since $f$ is the only Riemann
  theta function for $L$ up to scaling, and since $(-1)\colon A\to A$ fixes the
  identity, it follows that $(-1)^*f = f$.  Hence
  $(-1)^*T_x^*f = T_{-x}^*(-1)^*f = T_{-x}^*f$, so $(-1)^*h = h$.
\end{proof}

In particular, if $f_{\trop}$ is the tropical Riemann theta function for
$\Sigma$, then there exists a rational function $h$ on $A'$
such that
\[ h_{\trop}(v) = f_{\trop}(v+w) + f_{\trop}(v-w) - 2f_{\trop}(v), \]
where $w = \trop(x)$.

\bibliographystyle{egabibstyle}
\bibliography{papers}

\vskip 1cm

\end{document}